\documentclass[12pt,a4paper]{amsart}
\usepackage{xspace}
\usepackage[T1]{fontenc}
\usepackage{lmodern}

\usepackage{hyperref}
\usepackage{xcolor}
\hypersetup{%
    colorlinks=true,
    pdfborderstyle={/S/U/W 1}
}

\usepackage{enumitem}
\setenumerate[1]{label=(\arabic{*}), ref=(\arabic{*})}

\newcommand{\dirkP}{\ensuremath{(\mathfrak{D})}\xspace}
\newtheorem{thm}{Theorem}[section]
\newtheorem{cor}[thm]{Corollary}
\newtheorem{lem}[thm]{Lemma}
\newtheorem{prop}[thm]{Proposition}

\newtheorem{prob}{\sc Problem}
\newtheorem{eks}[thm]{\sc Example}
\theoremstyle{definition}
\newtheorem{defn}[thm]{Definition}
\theoremstyle{remark}
\newtheorem{rem}[thm]{Remark}

\numberwithin{equation}{section}
\newcommand{\eps}{\varepsilon}

\DeclareMathOperator{\linspan}{span}
\DeclareMathOperator{\cspan}{\overline{span}}
\DeclareMathOperator{\conv}{conv}
\DeclareMathOperator{\clco}{\overline{conv}}
\DeclareMathOperator{\sign}{sign}
\DeclareMathOperator{\supp}{supp}
\title{Delta- and Daugavet-points in Banach spaces}
\author{T. A. Abrahamsen, R. Haller, V. Lima, and K. Pirk}

\address[T.~A.~Abrahamsen]{Department of Mathematics, University of
  Agder, Postboks 422, 4604 Kristiansand, Norway.}
\email{trond.a.abrahamsen@uia.no}
\urladdr{http://home.uia.no/trondaa/index.php3}

\address[R.~Haller and K.~Pirk]{Institute of Mathematics, University of Tartu, J.~Liivi 2, 50409 Tartu, Estonia}
\curraddr{}
\email{rainis.haller@ut.ee, katriinp@ut.ee}

\address[V.~Lima]{Department of Engineering Sciences, University of Agder,
Postboks 422, 4604 Kristiansand, Norway.}
\email{Vegard.Lima@uia.no}

\thanks{R.~Haller and K.~Pirk were partially supported by institutional research funding IUT20-57 of the Estonian Ministry of Education and Research.}



\begin{document}
\begin{abstract}
  A $\Delta$-point $x$ of a Banach space is a norm one element
  that is arbitrarily close to convex combinations
  of elements in the unit ball that are almost at distance
  $2$ from $x$.
  If, in addition, every point in the unit ball is arbitrarily
  close to such convex combinations, $x$ is a Daugavet-point.
  A Banach space $X$ has the Daugavet
  property if and only if every norm one element is a Daugavet-point.

  We show that $\Delta$- and Daugavet-points
  are the same in $L_1$-spaces, $L_1$-preduals,
  as well as in a big class of M{\"u}ntz spaces.
  We also provide an example of a Banach
  space where all points on the unit sphere are $\Delta$-points,
  but where none of them are Daugavet-points.

  We also study the property that the unit ball
  is the closed convex hull of its $\Delta$-points.
  This gives rise to a new diameter two property
  that we call the convex diametral diameter two property.
  We show that all $C(K)$ spaces, $K$ infinite compact Hausdorff,
  as well as all M{\"u}ntz spaces have this property.
  Moreover, we show that this property is stable under absolute sums.
\end{abstract}

\maketitle

\section{Introduction}
\label{sec:introduction}
Let $X$ be a real Banach space with unit ball $B_X$,
unit sphere $S_X$, and dual $X^\ast$.
Recall that $X$ has the
\emph{local diameter two property (LD2P)}
if every slice of $B_X$ has diameter two.
Recall that a \emph{slice} of $B_X$ is a subset of the
form
\begin{align*}
  S(x^*,\eps) = \{x \in B_X: x^*(x) > 1 - \eps\},
\end{align*}
where $x^* \in S_{X^*}$ and $\eps > 0$.

For $x \in S_X$ and $\varepsilon > 0$ denote
\begin{equation*}
  \Delta_\varepsilon (x) =
  \{y \in B_X \colon \|x-y\| \geq 2 - \varepsilon \}.
\end{equation*}
We say that $x \in S_X$ is a \emph{$\Delta$-point} if
we have $x\in \clco \Delta_\eps(x)$,
the norm closed convex hull of $\Delta_\eps(x)$, for all $\eps > 0$.
The set of all $\Delta$-points in $S_X$ is denoted by $\Delta$,
i.e.
\begin{align*}
  \Delta =
  \{ x \in S_X \colon x \in \clco\Delta_\varepsilon(x)
   \mbox{ for all } \varepsilon > 0 \}.
\end{align*}
We will sometimes need to clarify which Banach space
we are working with and write
$\Delta_\varepsilon^X(x)$ and $\Delta_X$
instead of $\Delta_\varepsilon(x)$ and $\Delta$ respectively.

The starting point of this research was the discovery that if a Banach
space $X$ satisfies $B_X = \clco \Delta$, then $X$
has the LD2P.

We study spaces that satisfy the property $B_X = \clco \Delta$ in
Section~\ref{sec:convex-dld2p}. The case $S_X = \Delta$, i.e. $x \in
\clco\Delta_\varepsilon(x)$ for all $x\in S_X$ and $\varepsilon>0$,
has already appeared in the literature, but under different names:
The diametral local diameter two property (DLD2P) (\cite{BLR_diametral}),
the LD2P+ (\cite{MR3499106} and \cite{MR3574595}),
and space with bad projections (\cite{Zbl 1071.46015}).
We will use the term DLD2P in this paper.
From \cite[Corollary~2.3 and (7) p.~95]{MR1856978}
and \cite[Theorem~1.4]{Zbl 1071.46015} the following
characterization is known.

\begin{prop}\label{prop:delta-char}
  Let $X$ be a Banach space. The following assertions are
  equivalent:
  \begin{enumerate}
  \item
    $X$ has the DLD2P;
  \item
    for all $x\in S_X$ we have
    $x \in \clco\Delta_\varepsilon(x)$
    for all $\varepsilon > 0$;
  \item for all projections
    $P:X \to X$ of rank-1 we have $\|Id - P\| \geq 2$.
  \end{enumerate}
\end{prop}

Related to the DLD2P is the Daugavet property.
We have (cf. \cite[Corollary~2.3]{MR1856978}):

\begin{prop}\label{prop:daug-char}
  Let $X$ be a Banach space. The following assertions are
  equivalent:
  \begin{enumerate}
  \item
    $X$ has the Daugavet property, i.e.
    for all bounded linear rank-1 operators $T:X \to X$
    we have $\|Id - T\| = 1 + \|T\|$;
  \item for all $x \in S_X$ we have
    $B_X = \clco \Delta_\varepsilon(x)$
    for all $\varepsilon > 0$.
  \end{enumerate}
\end{prop}
  
Clearly the Daugavet property implies the DLD2P, but the converse
is not true \cite[Corollary~3.3]{Zbl 1071.46015}.

We will say that $x \in S_X$ is a \emph{Daugavet-point} if
we have $B_X = \clco \Delta_\eps(x)$ for all $\eps > 0$.
Every Daugavet-point is a $\Delta$-point, but the converse might fail
(see Example~\ref{exmp:Daug-vs-Delta} for an extreme example of this).

In our language, \cite[(7) p.~95]{MR1856978} states that
a Banach space $X$ the DLD2P is equivalent to:
\begin{itemize}
\item[\dirkP]
  For all projections
  $P:X \to X$ of rank-1 and \emph{norm-1} we have
  $\|Id - P\| = 2$.
\end{itemize}
This statement is repeated in \cite[Theorem~3.2]{MR3574595} and
used in the argument of
\cite[Theorem~3.5 (i) $\Leftrightarrow$ (iii)]{MR3574595}.
In the case of the Daugavet property, it \emph{is enough} to
consider only norm 1 operators $T$. This follows by scaling
(see the argument below Definition~2.1 in \cite{MR1856978}).
However, a scaled projection is not a projection.
Upon request, neither the authors of \cite{MR3574595}
nor \cite{MR1856978} have been able to give a correct
proof that \dirkP is equivalent to the DLD2P.
Thus the validity of this equivalence is still an open question.

Through an investigation of $\Delta$- and Daugavet-points
in concrete spaces, we have been able to show that
for $L_1(\mu)$-spaces, where $\mu$ is a $\sigma$-finite measure
on an infinite set, and for $L_1(\mu)$-predual spaces,
the property in \dirkP \emph{is} equivalent
to the DLD2P (and even to the Daugavet property)
(see Theorems~\ref{thm:Daugavet-eq-dirk-prop-L1}
and \ref{thm:Daugavet-eq-dirk-prop-L1predual} below).

In connection with the open problem just mentioned,
it is worth noting that for $X = \ell_1$ a pointwise version of the property
\dirkP holds for some $x \in S_X$ even though $S_X$ has no $\Delta$-points
(see Proposition \ref{prop:point-dirk-not-point-dld2p} and Theorem~\ref{thm:L1mu-DaugDeltaChar}).

In the following we will bring in our main results. In
Section~\ref{sec:daug-points-delta-different} we look at the $\Delta$-
and Daugavet-points in $L_1(\mu)$ spaces when $\mu$ is a $\sigma$-finite
measure, preduals of $L_1(\mu)$ spaces for such measure $\mu$,
and a big class of M\"{u}ntz spaces. We prove that $\Delta$-
and Daugavet-points are the same in all these cases
(see Theorems
\ref{thm:L1mu-DaugDeltaChar}, \ref{thm:Lindenstrauss-D-pt}, and
\ref{thm:muntz-DaugDeltaChar}).

In Section~\ref{sec:stability-results} we show that there are absolute
normalized norms $N$, different from the $\ell_1$- and $\ell_\infty$-norms,
for which $X\oplus _N Y$ has Daugavet-points, and also such $N$ for which
$X\oplus_N Y$ fails to have Daugavet-points.

In Section~\ref{sec:convex-dld2p} we introduce the convex
diametral diameter two property (convex DLD2P) defined naturally using
$\Delta$-points. We show that this property lies strictly between the
properties DLD2P and LD2P (see Corollary~\ref{cor:convDLD2Pisnew}).
We give examples of
classes of spaces with the convex DLD2P, more precisely we show that
all $C(K)$ spaces, $K$ infinite compact Hausdorff, as well as all
M\"{u}ntz spaces, have this property (see Proposition~\ref{prop:CK-convDLD2P} and
Theorem~\ref{Mutnz space has cDLD2P}). We also prove that if $X$ and
$Y$ have the convex DLD2P, then the sum $X\oplus_N Y$ has this
property whenever $N$ is an absolute normalized norm (see Theorem~\ref{sum
  has convDLD2P}).

%

\section{Preliminaries}
\label{sec:preliminaries}

We start this section collecting some characterizations of
$\Delta$- and Daugavet-points from the literature.

\begin{lem}\label{Delta point crit}
  Let $X$ be a Banach space and $x \in S_X$.
  The following assertions are equivalent:
  \begin{enumerate}
  \item\label{aaa}
    $x$ is a $\Delta$-point,
    that is $x \in \clco \Delta_\varepsilon(x)$
    for every $\varepsilon > 0$;
  \item\label{bbb}
    for every slice $S$ of $B_X$ with $x\in S$ and for every
    $\varepsilon > 0$ there exists $y \in S_X$
    such that $\|x - y\| \geq 2 - \varepsilon$;
  \item\label{ccc}
    for every $x^\ast \in X^\ast$ with $x^\ast(x) = 1$
    the projection $P = x^\ast \otimes x$
    satisfies $\|Id - P\| \geq 2$.
  \end{enumerate}
\end{lem}

\begin{proof}
  The equivalence of \ref{aaa} $\Leftrightarrow$ \ref{bbb} is
  proved using Hahn-Banach separation.

  The equivalence
  \ref{bbb} $\Leftrightarrow$ \ref{ccc}
  is a pointwise version of
  \cite[Theorem~1.4]{Zbl 1071.46015}
  and the same proof works.
\end{proof}

\begin{lem}\label{Daugavet point crit}
  Let $X$ be a Banach space and $x \in S_X$.
  The following assertions are equivalent:
  \begin{enumerate}
  \item\label{aaaa}
    $x$ is a Daugavet point,
    that is $B_X = \clco \Delta_\varepsilon(x)$
    for every $\varepsilon > 0$;
  \item\label{bbbb}
    for every slice $S$ of $B_X$ and for every $\varepsilon>0$
    there exists $y\in S$ such that $\|x - y\| \geq 2-\varepsilon$;
  \item\label{cccc}
    for every nonzero $x^\ast\in X^\ast$, the rank-1 operator
    $T = x^\ast\otimes x$ satisfies $\|Id - T\| = 1 + \|T\|$;
  \item\label{dddd}
    for every $x^\ast\in S_{X^\ast}$ the rank-1 norm-1 operator
    $T = x^\ast \otimes x$ satisfies $\|Id - T\|= 2$.
  \end{enumerate}
\end{lem}

\begin{proof}
  The equivalence \ref{bbbb} $\Leftrightarrow$ \ref{cccc}
  is a pointwise version of Lemma~2.2 in \cite{MR1621757}.
  The equivalence \ref{aaaa} $\Leftrightarrow$ \ref{bbbb}
  follows by Hahn-Banach separation as observed
  by \cite[Corollary~2.3]{MR1856978}.

  While \ref{cccc} $\Rightarrow$ \ref{dddd} is trivial
  the implication \ref{dddd} $\Rightarrow$ \ref{cccc}
  follows by scaling as explained in the paragraph
  following Definition~2.1 in \cite{MR1856978}.
\end{proof}

The next proposition shows that we cannot add
a version of Lemma~\ref{Daugavet point crit}~\ref{dddd}
to Lemma~\ref{Delta point crit}.
In fact, we will see in Theorem~\ref{thm:L1mu-DaugDeltaChar}
that no point on the sphere in $\ell_1$ is
a $\Delta$-point.

\begin{prop}\label{prop:point-dirk-not-point-dld2p}
  Let $X = \ell_1$ and $x=(x_i)_{i=1}^\infty \in S_X$ a smooth point
  with $|x_1| > 1/3.$
  Then:
  \begin{enumerate}
  \item
    for $x^* \in S_{X^*}$ with $x^*(x) = 1$, the projection
    $P = x^* \otimes x$ satisfies $\|Id - P\| = 2$;
  \item
    the projection $P = x_1^{-1}e_1^* \otimes x$
    satisfies $\|Id - P\|< 2$.
  \end{enumerate}
\end{prop}

\begin{proof}
  Write $x = (x_i)_{i=1}^\infty$.
  Let
  $x^* := (\sign x_i)_{i=1}^\infty \in S_{X^*}$
  and $P := x^* \otimes x$.
  Observe that $x^*(x) = 1$.
  If $e_n$ is the $n$'th standard basis vector in $X$, then
  \begin{align*}
    \|(Id - P)(e_n)\|
    &= \|e_n - \sign x_n x\|
    = |1 - (\sign x_n) x_n| + \sum_{i \neq n} |x_i| \\
    &= 1 - |x_n| + \|x\| - |x_n| = 2 - 2|x_n|,
  \end{align*}
  and since this holds for all $n$ we get $\|Id - P\| = 2$.

  Let $P := x_1^{-1}e_1^* \otimes x$,
  where $e_i^*$ is the $i$'th coordinate vector in $X^* = \ell_\infty$.
  Observe that $x_1^{-1}e_1^*(x) = 1$, so that $P$ is a projection.
  If $y \in S_X$ we get
  \begin{align*}
    \|(Id-P)y\|
    &
      = \|y - x_1^{-1}y_1x\|
      = \sum_{i > 1}|y_i - x_1^{-1}y_1x_i|\\
    &\le
      \sum_{i > 1}|y_i| + |x_1|^{-1}|y_1| \sum_{i > 1} |x_i| \\
    &
      = 1 - 2|y_1| + |x_1|^{-1}|y_1|
      \le 1 + \big|2 - |x_1|^{-1}\big| < 2,
  \end{align*}
  so $\|Id - P\| < 2$, and we are done.
\end{proof}

Let us note that both the DLD2P and property \dirkP
pass from the dual to the space.

\begin{prop}
  Let $X$ be a Banach space.
  Then:
  \begin{enumerate}
  \item
  if $X^*$ has the DLD2P, then $X$ has the DLD2P;
  \item
  if $\|Id_{X^*} - P\| = 2$ for all norm-1 rank-1 projections $P$
  on $X^*$, then $\|Id_X - Q\| = 2$ for all norm-1 rank-1
  projections $Q$ on $X$.
  \end{enumerate}
\end{prop}

\begin{proof}
  The second statement is trivial, while
  the first one only requires a bit of rewriting:
  If $Q$ is a rank-1 projection on $X$,
  then $Q = x^{*} \otimes x$ with
  $x^{*} \in X^{*}$, $x \in S_{X}$,
  and $x^{*}(x) = 1$.
  Then
  \begin{equation*}
    P = Q^* = x \otimes x^* = (\|x^*\| x) \otimes \frac{x^*}{\|x^*\|}
  \end{equation*}
  is a rank-1 projection on $X^*$ and by assumption
  $\|Id_{X^*} - P\| = \|Id_{X} - Q\| \ge 2$.
\end{proof}

As we noted in the Introduction, we do not know
if the property in \dirkP is equivalent to
the DLD2P. We end this section by observing
that, just like the DLD2P, property \dirkP
implies that all slices of the unit ball of
both the space and its dual have diameter two.
(See \cite[Theorem~1.4]{Zbl 1071.46015}
and \cite[Theorem~3.5]{MR3574595} for the corresponding
DLD2P result.)
The following result also shows that
despite of Proposition~\ref{prop:point-dirk-not-point-dld2p},
$\ell_1$ is not a candidate for separating property \dirkP
and the DLD2P since $\ell_1$ does not have the LD2P.

\begin{prop}\label{prop:dirk-prop-implies-LD2Ps}
  Let $X$ be a Banach space.
  If $\|Id - P\| = 2$ for all norm-1 rank-1 projections $P$
  on $X$, then $X$ has the LD2P and $X^*$ has
  the $w^*$-LD2P.
\end{prop}

\begin{proof}
  Let $x^* \in S_{X^*}$ and $\varepsilon > 0$
  define a slice $S(x^*,\varepsilon)$.
  Let $\delta > 0$ such that $\delta < \frac{\varepsilon}{2}$.
  Find $y^* \in S_{X^*}$ such that $y^*$ attains
  its norm on $B_X$ and $\|x^* - y^*\| < \frac{\varepsilon}{2}$.
  Let $y \in B_X$ be such that $y^*(y) = 1$
  and define $P = y^* \otimes y$.
  Then $\|Id - P\| = 2$ by assumption and
  we can find $z \in S_X$ such that
  \begin{equation*}
    \|z - P(z)\|
    = \|z - y^*(z)y\|
    > 2 - \delta.
  \end{equation*}
  We may assume that $y^*(z) > 0$.
  We have
  \begin{equation*}
    y^*(z) = |y^*(z)| = \|P(z)\|
    \ge \|P(z) - z\| - \|z\|
    > 2 - \delta - 1
    > 1 - \frac{\varepsilon}{2}.
  \end{equation*}
  Hence
  \begin{equation*}
    x^*(z) = y^*(z) - (y^*-x^*)(z)
    > 1 - \frac{\varepsilon}{2} - \frac{\varepsilon}{2}
    = 1 - \varepsilon,
  \end{equation*}
  i.e. $z \in S(x^*,\varepsilon)$, and
  \begin{equation*}
    \|z - y\| \ge \|z - y^*(z)y\| - \|y^*(z)y - y\|
    > 2 - \delta - |y^*(z) - 1|
    > 2 - 2\delta.
  \end{equation*}
  This proves that $X$ has the LD2P.

  To show that $X^*$ has the $w^*$-LD2P we start with
  a $w^*$-slice $S(x,\varepsilon)$, where $x \in S_X$
  and $\varepsilon > 0$.
  Then we find a $y^* \in S_{X^*}$ where
  $\|Id^* - P^*\|$ almost attains its norm.
  The proof is similar to the LD2P case.
\end{proof}

\section{$\Delta$- and Daugavet-points for different classes of spaces}
\label{sec:daug-points-delta-different}

In the first two parts of this section we study $\Delta$- and
Daugavet-points in Banach spaces $X$ of the type $L_1(\mu)$, $C(K)$,
and $L_1(\mu)$-preduals. Crucial in our study is the discovery that a
$\Delta$-point $f \in S_X$ can be characterized in terms of
properties of the support of $f$ (see
Theorems~\ref{thm:L1mu-DaugDeltaChar} and
\ref{thm:CK-Daug-Delta-limit}). These characterizations of being a
$\Delta$-point are easy to check, and we use them to prove that
$\Delta$- and Daugavet-points are in fact the same in all such spaces $X$.
For example, if $X = C([0,\omega]) = c$ then the
Daugavet-points are exactly the sequences with limits $\pm
1$.

In the last part of the section we study $\Delta$- and
Daugavet-points in M{\"u}ntz spaces $X$ of the type $M_0(\Lambda)
\subset M(\Lambda) \subset C[0,1]$ (see Subsection
\ref{sec:muntz-space} for a definition of a M{\"u}ntz space).
Our initial motivation for doing this, was the
known fact that such spaces $X$ are isomorphic, even almost isometrically
isomorphic in the case $X = M_0(\Lambda)$, to subspaces of $c$ (see
\cite{werner-muntz} and \cite{M1}). Based on this, the results from
\cite{ALMN}, and other results from \cite{M1} one could  expect
similar results for M{\"u}ntz spaces as for $c$. And, indeed, this
is the case, at least for $X = M_0(\Lambda)$ (see Theorem
\ref{thm:muntz-DaugDeltaChar}). In this class of M{\"u}ntz spaces the
$\Delta$- and Daugavet-points are the same and the Daugavet-points
are exactly the functions $f \in S_X$ for which $f(1) = \pm 1.$

\subsection{$L_1(\mu)$ spaces}
\label{sec:daug-delta-points-L1}
Let $\mu$ be a (countably additive, non-negative) measure
on some $\sigma$-algebra $\Sigma$ on a set $\Omega$.
We will assume that $\mu$ is $\sigma$-finite even
though it is not strictly necessary in all the results.
As usual an \emph{atom for $\mu$} is a set $A \in \Sigma$
such that $0 < \mu(A) < \infty$, and
if $B \in \Sigma$ with $B \subseteq A$
satisfies $\mu(B) < \mu(A)$, then $\mu(B) = 0$.

In this section we consider the space
$L_1(\mu) = L_1(\Omega,\Sigma,\mu)$.

\begin{thm}\label{thm:L1mu-DaugDeltaChar}
  The following assertions for $f \in S_{L_1(\mu)}$ are equivalent:
  \begin{enumerate}
  \item\label{item:L1mu-1}
    $f$ is a Daugavet point;
  \item\label{item:L1mu-2}
    $f$ is a $\Delta$-point;
  \item\label{item:L1mu-3}
    $\supp(f)$ does not contain an atom for $\mu$.
  \end{enumerate}
\end{thm}

\begin{proof}
  \ref{item:L1mu-1} $\Rightarrow$ \ref{item:L1mu-2}
  is trivial.

  \ref{item:L1mu-2} $\Rightarrow$ \ref{item:L1mu-3}.
  Fix $f \in S_{L_1(\mu)}$.
  Let $A$ be an atom in $\supp(f)$.
  Note that a measurable function is a.e. constant
  on an atom. We may assume that
  $f|_A = c$ a.e. for some positive constant $c$.
  Fix $0 < \varepsilon < 2c \mu(A)$.

  Let $g \in B_{L_1(\mu)}$ be such that
  $\|f - g\| \ge 2 - \varepsilon$.
  We have $g|_A = d$ for some constant $d$.
  Note that
  \begin{align*}
    2 - \varepsilon
    &\le \int_\Omega |f - g| d\mu
      = \int_{\Omega \setminus A} |f - g| d\mu
      + \int_{A} |f - g| d\mu \\
    &\le
      \int_{\Omega \setminus A} |f| d\mu
      + \int_{\Omega \setminus A} |g| d\mu
      + \int_{A} |f - g| d\mu \\
    &\le
      1 - \int_A |f| d\mu
      +
      1 - \int_A |g| d\mu
      +
      \int_A |f - g| d\mu \\
    &= 1 - c\mu(A) + 1 - |d| \mu(A)
      + |c-d| \mu(A).
  \end{align*}
  Therefore
  \begin{equation*}
    c \mu(A) + d\mu(A) \le |c-d|\mu(A) + \varepsilon.
  \end{equation*}
  If $c \le d$, then $|c-d| = d-c$
  and we get $c \le \frac{\varepsilon}{2\mu(A)}$,
  and this contradicts our choice of $\varepsilon$.
  Thus we have $c \ge d$, and hence
  $|c-d| = c-d$ and $d \le \frac{\varepsilon}{2\mu(A)} < c$.

  If $g_1, \dotsc, g_m \in \Delta_\varepsilon(f)$, then
  \begin{equation*}
    \|f - \sum_{i=1}^m \frac{1}{m} g_i\|
    \ge \int_A | f - \sum_{i=1}^m \frac{1}{m} g_i| \; d\mu
    \ge
    (c - \frac{\varepsilon}{2\mu(A)})\mu(A) > 0.
  \end{equation*}
  This shows that $f \notin \clco \Delta_\varepsilon(f)$
  for this choice of $\varepsilon$.

  \ref{item:L1mu-3} $\Rightarrow$ \ref{item:L1mu-1}.
  Let $f \in S_{L_1(\mu)}$ such that $\supp(f)$
  does not contain atoms.
  Let $\varepsilon > 0$, $\delta > 0$,
  and $x_0^* \in S_{{L_1(\mu)}^*}$.
  By Lemma~\ref{Daugavet point crit}
  we need to find $g \in S_{L_1(\mu)}$ with
  $\| f - g \| \ge 2 - \varepsilon$ such
  that $g \in S(x_0^*, \delta)$.

  Since $\mu$ is $\sigma$-finite
  (so that $L_1(\mu)^* = L_\infty(\mu)$)
  we can find a step-function
  $x^* = \sum_{i=1}^n a_i \chi_{E_i} \in S_{{L_1(\mu)}^*}$
  such that $\|x^* - x_0^*\| < \delta$
  (and the $E_i \cap E_j = \emptyset$ for $i \neq j$).

  We may assume that $|a_1| = 1$. Find subset a $A$ of $E_1$
  such that $\int_A |f| d\mu < \varepsilon/2$.
  Define
  \begin{equation*}
    g := \frac{\sign(a_1)}{\mu(A)}\chi_A \in S_{L_1(\mu)}.
  \end{equation*}
  Then
  \begin{equation*}
    x^*(g) = \sum_{i=1}^n \int_{E_i} a_i g d\mu
    = \frac{1}{\mu(A)} \int_A a_1 \sign(a_1) d\mu = 1,
  \end{equation*}
  \begin{equation*}
    \|f - g\| = \int_{A^c} |f| d\mu
    + \int_A |f-g| d\mu
    \ge |f| + |g| - 2\int_A |f| d\mu
    \ge 2 - \varepsilon,
  \end{equation*}
  and finally
  \begin{equation*}
    x_0^*(g) = x^*(g) - (x^*-x_0^*)(g)
    > 1 - \delta
  \end{equation*}
  as desired.
\end{proof}

\begin{lem}\label{lem:L1-atom-not-LD2P}
  If $\mu$ is a measure with an atom, then
  $L_1(\mu)$ does not have the LD2P.
\end{lem}

\begin{proof} 
  Assume that $A$ is an atom and
  consider $\chi_A \in L_1(\mu)^*$.
  We have $\|\chi_A\| = 1$.
  If $f \in S(B_{L_1(\mu)}, \chi_A, \varepsilon)$, then
  \begin{equation*}
    f(t) > \frac{1-\varepsilon}{\mu(A)}
    \qquad \mbox{for a.e. } t \in A,
  \end{equation*}
  and
  \begin{equation*}
    f(t) \le \frac{1}{\mu(A)}
    \qquad \mbox{for a.e. } t \in A.
  \end{equation*}
  Hence $\|f|_A\| > 1 - \varepsilon$
  and $\|f|_{A^C}\| < \varepsilon$.

  Thus for $f_1, f_2 \in S(B_{L_1(\mu)}, \chi_A, \varepsilon)$
  we have
  \begin{align*}
    \|f_1 - f_2\|
    &\le
      \int_{A^c} |f_1 - f_2| \; d\mu
      +
      \int_{A} |f_1 - f_2| \; d\mu \\
    &\le
      \|f_1|_{A^c}\| + \|f_2|_{A^c}\|
      +
      \int_A \frac{\varepsilon}{\mu(A)} \; d\mu
      \le 3 \varepsilon,
  \end{align*}
  so this slice does not have diameter $2$.
\end{proof}

\begin{thm}\label{thm:Daugavet-eq-dirk-prop-L1}
  Consider $X = L_1(\mu)$.
  The following assertions are equivalent:
  \begin{enumerate}
  \item\label{item:L1-1}
    $\|Id - P\| = 2$ for all norm-1 rank-1
    projections on $X$;
  \item\label{item:L1-2}
    $X$ has the Daugavet property.
  \end{enumerate}
\end{thm}

\begin{proof}
  If \ref{item:L1-1} holds, then
  $X$ has the LD2P by Proposition~\ref{prop:dirk-prop-implies-LD2Ps}.
  From Lemma~\ref{lem:L1-atom-not-LD2P} we see that
  $X$ does not have atoms.
  By \cite{MR2146042}
  (see also \cite{MR2326380} for the explicit statement
  for $L_1(\mu)$ spaces)
  $X$ has the Daugavet property.

  The other direction is trivial.
\end{proof}

%
%
\subsection{$C(K)$ and $L_1(\mu)$-predual spaces}
\label{sec:daug-delta-points-CK}
In the following we explore the $\Delta$- and Daugavet-points in the class of $L_1(\mu)$-predual spaces and $C(K)$ spaces. We start with a characterization of both
Daugavet and $\Delta$-points in $C(K)$ spaces.

\begin{thm}\label{thm:CK-Daug-Delta-limit}
  Let $K$ be an infinite compact Hausdorff space. The following
  assertions for $f\in S_{C(K)}$ are equivalent:
  \begin{enumerate}
  \item\label{item:CK-DDlimit-1}
    $f$ is a Daugavet point;
  \item\label{item:CK-DDlimit-2}
    $f$ is a $\Delta$-point;
  \item\label{item:CK-DDlimit-3}
    $\|f\| = |f(x_0)|$ for a limit point $x_0$ of $K$.
  \end{enumerate}
\end{thm}

\begin{proof}
  \ref{item:CK-DDlimit-1} $\Rightarrow$ \ref{item:CK-DDlimit-2}
  is trivial.

  \ref{item:CK-DDlimit-3} $\Rightarrow$ \ref{item:CK-DDlimit-1}.
  Let $f\in S_{C(K)}$ and assume that there is a limit point
  $x_0$ of $K$ such that $|f(x_0)| = 1$.
  We will show that $f$ is a Daugavet-point.
  Fix $g \in B_X$, $\varepsilon>0$, and $m \in \mathbb N$.
  Consider a neighbourhood $U$ of $x_0$
  such that $|f(x_0) - f(x)| < \varepsilon$
  for every $x\in U$.
  Since $x_0$ is a limit point, we can find $m$ different points
  $x_1, \dotsc, x_m\in U$ and corresponding pairwise disjoint neighbourhoods
  $U_1, \dotsc, U_m\subset U$.
  For every $1\leq i\leq m$ use Urysohn's lemma to find
  a continuous function $\eta_i \colon K \to [0,1]$ with
  $\eta_i(x_i) = 1$ and $\eta_i = 0$
  on $K \setminus U_i$.
  Define $g_i\in B_{C(K)}$ by
  \begin{equation*}
    g_i(x) = \big(1-\eta_i(x)\big)g(x) - \eta_i(x)f(x_0).
  \end{equation*}
  From $g_i(x_i) = -f(x_0)$ it follows that
  \begin{equation*}
    \|f - g_i\| \ge |f(x_i) - g(x_i)|
    = |f(x_i) + f(x_0)|
    > 2 - \varepsilon.
  \end{equation*}
  Hence $g_i \in \Delta_\varepsilon(f)$.
  Note that $g - g_i = 0$ on $K \setminus U_i$,
  and consequently
  \begin{equation*}
    \|g - \frac{1}{m}\sum_{i=1}^m g_i\|
    \le
    \frac{1}{m}\max_{1\leq i\leq m} \|g-g_i\|
    \le \frac{2}{m}.
  \end{equation*}
  We thus get $g \in \clco\Delta_\varepsilon(f)$,
  and so $f$ is a Daugavet-point.

  \ref{item:CK-DDlimit-2} $\Rightarrow$ \ref{item:CK-DDlimit-3}.
  We assume that there is no limit point $x$ of $K$ such that
  $|f(x)| = 1$ and show that $f$ is not a $\Delta$-point.
  Define
  \begin{equation*}
      H := \{ x \in K \colon |f(x)| = 1 \}.
  \end{equation*}
  Then $H$ is a set of isolated points.
  By compactness,
  $H$ is finite since otherwise it would contain a limit point.
  Note that $H$ is (cl)open hence
  $\delta = 1 - \max_{x\in K \setminus H}|f(x)| > 0$.
  Let $\varepsilon_h := \sign f(h)$ for all $h \in H$.
  Since $H \neq \emptyset$ we can define
  \begin{equation*}
     \mu = \frac{1}{|H|}\sum_{h\in H} \varepsilon_h \delta_h,
  \end{equation*}
  where $\delta_h \in S_{C(K)^*}$ is the point evaluation map at
  $h$. We have $\|\mu\| = 1$ and $\langle \mu, f \rangle = 1$,
  hence $P = \mu \otimes f$ is a norm 1 projection.

  Let $g \in B_{C(K)}$ and consider
  $\| (Id - P) g\| = \|g -Pg\|
  = \|g - \langle \mu, g \rangle f\|$.
  For $x \notin H$ we have
  \begin{equation*}
    |g(x) - \langle \mu,g \rangle f(x)|
    \leq 1 + 1 - \delta = 2 - \delta.
  \end{equation*}
  While for $x\in H$ we use that
  \begin{equation*}
    \langle \mu, g\rangle
    = \frac{1}{|H|} \sum_{h\in H} \varepsilon_h g(h)
  \end{equation*}
  and $\varepsilon_h f(h) = |f(h)| = 1$, so that
  \begin{align*}
    |g(x) - \langle \mu, g\rangle f(x)|
    &=
      |g(x) - \frac{1}{|H|} \sum_{h\in H} \varepsilon_h g(h) f(x)| \\
    &=
      \Big|
      \Big(
      1 - \frac{1}{|H|}
      \Big) g(x)
      -
      \frac{1}{|H|}
      \sum_{h \in H \setminus \{x\}} \varepsilon_h g(h) f(x)
      \Big| \\
    &\leq
      \Big(1 - \frac{1}{|H|} \Big)
      +
      \frac{|H| - 1}{|H|}
      = 2 - \frac{2}{|H|}.
  \end{align*}
  With $\varepsilon = \min\{ \delta, 2/|H| \}$
  we have
  $\|(Id - P) g\| \leq 2 - \varepsilon < 2$
  for all $g\in B_{C(K)}$ hence $\|Id - P\| < 2$.
\end{proof}

%
%
%
%
%
Let $X$ be a Banach space such that
$X^*$ is isometric to an $L_1(\mu)$-space,
that is, $X$ is a Lindenstrauss space.
For such spaces we have $X^{**}$
is isometric to the space $C(K)$ for some
(extremally disconnected) compact Hausdorff space $K$
(see \cite[Theorem~6.1]{LinMem}).
Our next goal is to show that for such spaces
$\Delta$- and Daugavet-points are the same.
We first need a lemma.

\begin{lem}\label{lem:Daug-ai-ideal}
  Let $X$ be a Banach space and let $x,y \in S_X$.
  The following assertions are equivalent:
  \begin{enumerate}
  \item\label{item:Dai-1}
    $y \in \clco \Delta_\varepsilon^X(x)$ for all $\varepsilon > 0$;
  \item\label{item:Dai-2}
    $y \in \clco \Delta_\varepsilon^{X^{**}}(x)$ for all
    $\varepsilon > 0$.
  \end{enumerate}
\end{lem}

\begin{proof}
  \ref{item:Dai-1} $\Rightarrow$ \ref{item:Dai-2}
  is trivial as $\Delta^X_\varepsilon(x) \subset \Delta^{X^{**}}_\varepsilon(x)$.

  \ref{item:Dai-2} $\Rightarrow$ \ref{item:Dai-1}.
  Let $\varepsilon > 0$ and $\delta > 0$.
  Find $y_n^{**} \in B_{X^{**}}$ such that
  $\|x - y_n^{**}\| \geq 2 - \varepsilon$ and
  $\|y - \sum_{n=1}^m \lambda_n y_n^{**}\| < \delta$.

  Define $E := \linspan\{x,y,y_n^{**}\}$.
  Let $\eta > 0$ and
  use the principle of local reflexivity
  to find $T: E \to X$ such that
  \begin{itemize}
  \item[(i)] $T(e) = e$ for all $e \in E \cap X$.
  \item[(ii)] $(1-\eta)\|e\| \le \|Te\| \le (1+\eta)\|e\|$.
  \end{itemize}
  Then $\|x - T y_n^{**}\| = \|T(x - y_n^{**})\|
  \ge (1-\eta)\|x - y_n^{**}\| > 2 - \varepsilon$
  if $\eta$ is small enough.
  Also, if $\eta$ is small enough,
  \begin{equation*}
    \|y - \sum_{n=1}^m \lambda_n T y_n^{**}\|
    \le (1+\eta)
    \|y - \sum_{n=1}^m \lambda_n y_n^{**}\|
    < \delta.
  \end{equation*}
\end{proof}

\begin{rem}
  \label{rem:ai-ideal-Deltaset}
  The argument shows that the conclusion in
  Lemma~\ref{lem:Daug-ai-ideal} also holds in the more general setting
  of $X$ being an almost isometric ideal (see \cite{ALN2} for a
  definition) in $Z$, replacing $X^{**}$ with $Z$.
\end{rem}

\begin{thm}\label{thm:Lindenstrauss-D-pt}
  Let $X$ be an (infinite dimensional)
  $L_1(\mu)$-predual and $x \in S_X$.
  The following assertions are equivalent:
  \begin{enumerate}
  \item\label{item:L-D-1}
    $x$ is a $\Delta$-point;
  \item\label{item:L-D-2}
    $x$ is a Daugavet point.
  \end{enumerate}
  %
\end{thm}

\begin{proof}
  \ref{item:L-D-1} $\Rightarrow$ \ref{item:L-D-2}.
  By Lemma~\ref{lem:Daug-ai-ideal}
  we get
  $x \in \clco \Delta^{X^{**}}_\varepsilon(x)$
  for all $\varepsilon > 0$.
  Since $X^{**}$ is isometric to a $C(K)$-space,
  we get from Theorem~\ref{thm:CK-Daug-Delta-limit} that
  $x$ is a Daugavet-point in $X^{**}$,
  that is,
  $B_{X^{**}} = \clco \Delta^{X^{**}}_\varepsilon(x)$
  for all $\varepsilon > 0$.
  Using Lemma~\ref{lem:Daug-ai-ideal}
  again we get the desired conclusion.

  \ref{item:L-D-2} $\Rightarrow$ \ref{item:L-D-1}
  is trivial.
\end{proof}

\begin{thm}\label{thm:Daugavet-eq-dirk-prop-L1predual}
  Let $X$ be an $L_1(\mu)$-predual.
  The following assertions are equivalent:
  \begin{enumerate}
  \item\label{item:Linden-dirkP-1}
    $\|Id - P\| = 2$ for all norm-1 rank-1 projections $P$ on $X$;
  \item\label{item:Linden-dirkP-2}
    $X$ has the Daugavet property.
  \end{enumerate}
\end{thm}

\begin{proof}
  \ref{item:Linden-dirkP-2} $\Rightarrow$ \ref{item:Linden-dirkP-1}
  is trivial.

  \ref{item:Linden-dirkP-1} $\Rightarrow$ \ref{item:Linden-dirkP-2}.
  If $\|Id - P\| = 2$ for all norm-1 rank-1 projections,
  then $X^*$ has the $w^*$-LD2P by
  Proposition~\ref{prop:dirk-prop-implies-LD2Ps}
  which is equivalent to $X$ having extremely rough norm.
  By \cite[Theorem~2.4]{MR2326380}
  this implies the Daugavet property
  for $L_1(\mu)$-predual spaces.
\end{proof}
\subsection{M\"{u}ntz space}
\label{sec:muntz-space}
Now we explore $\Delta$- and Daugavet-points in the setting of
M\"{u}ntz spaces. Let us first clarify what we mean by such spaces.
\begin{defn} \label{defn: Mutnz space}
  Let $\Lambda = (\lambda_n)_{n=0}^\infty$
  be an increasing sequence of non-negative real numbers
  \begin{equation*}
    0 = \lambda_0 < \lambda_1 < \dotsb < \lambda_n < \dotsb
  \end{equation*}
  such that $\sum_{i=1}^{\infty} \frac{1}{\lambda_i} < \infty$.
  Then
  $M(\Lambda) := \cspan \{t^{\lambda_n}\}_{n=0}^\infty \subset C[0,1]$
  is called the M\"{u}ntz space associated with $\Lambda$.

  We will sometimes need to exclude the constants
  and consider the subspace
  $M_0(\Lambda) := \cspan \{t^{\lambda_n}\}_{n=1}^\infty$
  of $M(\Lambda)$.
\end{defn}

In order to prove a result about the Daugavet points in M\"{u}ntz
spaces, we need the following result.

\begin{lem}\label{lem:spike_functions}
  For all $\varepsilon > 0$ and $\delta > 0$,
  there exist $k,l \in \mathbb{N}$ with $k < l$
  such that for
  $f = (t^{\lambda_k} - t^{\lambda_l})/
  \|t^{\lambda_k} - t^{\lambda_l}\|$
  one has $f \ge 0$ and
  $\left. f \right|_{[0,1-\varepsilon]} < \delta$.
\end{lem}

\begin{proof}
  Fix positive numbers $\varepsilon$ and $\delta$. Let $k$ be such that
  \begin{equation*}
    t^{\lambda_k}|_{[0,1-\varepsilon]}<\dfrac\delta2.
  \end{equation*}
  Choose $l>k$ such that $\|t^{\lambda_k}-t^{\lambda_l}\|>1/2$.
  Then
  \begin{equation*}
    \dfrac{t^{\lambda_k}-t^{\lambda_l}}{\|t^{\lambda_k}-t^{\lambda_l}\|}
    <
    \dfrac{\delta/2}{1/2}
    =
    \delta
  \end{equation*}
  for any $t\in[0,1-\varepsilon]$.
\end{proof}
\begin{prop}\label{f at the end 1 is Daugavet point}
  Let $X = M(\Lambda)$ or $X = M_0(\Lambda)$.
  If $f \in S_X$ satisfies $f(1) = \pm 1$,
  then $f$ is a Daugavet point.
\end{prop}

\begin{proof}
  Fix $f\in S_X$ with $f(1) = \pm 1$ and $\varepsilon > 0$.
  We show that any $g \in S_X$ can be approximated by the
  elements of $\conv \Delta_\varepsilon(f)$.
  For this purpose, fix $g\in S_X$, $\delta > 0$ and
  choose $m \in \mathbb{N}$ with $m \geq 2/\delta$.

  Let $t_1 \in (0,1)$ be such that
  $|f(1) - f(t)| < \delta$ and $|g(1) - g(t)| < \delta$
  for all $t \in [t_1,1]$.
  We use Lemma~\ref{lem:spike_functions} to obtain
  $f_1$ such that $\left. f_1 \right|_{[0,t_1]} < \delta/2$.

  Let $t_2 \in (0,1)$ be such that
  $\left. f_1 \right|_{[t_2,1]} < \delta/2$.
  We use Lemma~\ref{lem:spike_functions} again
  to obtain $f_2$ such that $\left. f_2 \right|_{[0,t_2]} < \delta/2$.

  We continue finding $t_0 < t_1 < \cdots < t_m < t_{m+1} =: 1$
  and $f_1,\dotsc,f_m$.
  Define $g_i := g - [g(1)+1]f_i$ for $i = 1,\dotsc,m$.
  Then $\|g_i\| \leq 1 + \delta$. Indeed, for
  $t\in [0,1] \setminus [t_i, t_{i+1}]$
  we have that $f_i(t) < \delta/2$ and therefore
  \begin{equation*}
    |g_i(t)| \le |g(t)| + (1 + g(1))f_i(t)
    < 1 + 2\frac{\delta}{2} = 1 + \delta,
  \end{equation*}
  while for
  $t \in [t_i,t_{i+1}]$ we have
  \begin{align*}
    |g_i(t)| &\le |g(1) - [g(1)+1]f_i(t)| + |g(t) - g(1)| \\
    &\le |g(1)| (1-f_i(t)) + f_i(t) + \delta\\
    &\le 1 - f_i(t) + f_i(t) + \delta = 1 + \delta.
  \end{align*}
  
  Denote by $s_i$ the unique
  point in $(t_i,t_{i+1})$ where $f_i(s_i) = 1$.
  We have
  \begin{align*}
    \|g_i - f\| &\ge |g_i(s_i) - f(s_i)|\\
    & = |(g(s_i) - (g(1) + 1)) - f(s_i)| \\
    &\ge |1 + f(s_i)| - |g(1)-g(s_i)| \\
    &\ge 2 - \delta - \delta
    = 2 - 2\delta.
  \end{align*}
  Hence
  \begin{equation*}
    \|(1+\delta)^{-1}g_i - f\|
    \ge \|g_i - f\| - \|(1+\delta)^{-1} g_i - g_i\|
    \ge 2 - 3 \delta
  \end{equation*}
  since
  \begin{equation*}
    \|(1+\delta)^{-1} g_i - g_i\|
    = |(1+\delta)^{-1} - 1|\|g_i\|
    \le |(1+\delta)^{-1} - 1|(1+\delta)
    \le \delta.
  \end{equation*}
  We get that $(1+\delta)^{-1}g_i \in \Delta_\varepsilon(f)$
  whenever $3\delta < \varepsilon$.
  Finally
  \begin{align*}
    \|g - \sum_{i=1}^m \frac{1}{m}(1+\delta)^{-1} g_i\|
    &=
      \|(1-(1+\delta)^{-1})g + (1+\delta)^{-1}[g(1)+1]
      \sum_{i=1}^m \frac{1}{m} f_i \|
    \\
    &\le
      \frac{\delta}{1+\delta}\|g\|
      +
      \frac{(g(1)+1)}{m(1+\delta)}
      \| \sum_{i=1}^m f_i \|
    \\
    &\le
      \frac{\delta}{1+\delta}
      +
      \frac{2}{m}
      ( 1 + (m-1)\frac{\delta}{2} )
      \\
    &\le \delta + \delta + \delta
      \le 3\delta.
  \end{align*}
  Hence $g \in \clco \Delta_\varepsilon(f)$.
\end{proof}
\begin{prop}\label{prop:muntz-delta}
  Let $X$ be a M{\"u}ntz space $M_0(\Lambda)$ with $\lambda_1 \ge 1$.
  If $f \in S_X$ with $|f(1)| < 1$,
  then $f \notin \Delta$.
\end{prop}

\begin{proof}
  First note that from the full Clarkson-Erd{\"o}s-Schwartz theorem (see
  \cite{MR1961190}), $f$ is the restriction to $(0,1)$ of
  an analytic function on
  $\Omega = \{x \in \mathbb{C}\setminus (-\infty,0]: |z| < 1\}$.
  Let $I$ be the set of points in $[0,1]$ where $f$ attains its
  norm, and put $I^{\pm} = \{x \in I: f(x) = \pm 1\}$. From the
  assumptions we have $I \subset (0,1)$ since every $g \in M_0(\Lambda)$
  satisfies $g(0) = 0$. 

  Suppose $I$ is infinite. Then either $I^{+}$ or $I^-$ is
  infinite. Suppose without loss of generality that $I^+$ is. Then
  $I^+$ must have an accumulation point $a$ in $[0,1].$ By the
  continuity of $f$ we must have $f(a) = 1$, so
  $0 < a < 1.$ Since $f$ is analytic on $\Omega$ and $I^+$,
  $I^+$ has an accumulation point in $(0,1) \subset \Omega,$ we must
  have $1 - f = 0$ everywhere, which is a contradiction.

  Suppose $I$ is finite and that $f$ attains its norm on
  $(y_k)_{k=1}^m \subset (0,1)$ with
  $0 < y_1 < y_2 < \dotsb < y_m < 1$,
  i.e. $1 = \|f\| = |f(y_k)|$ for every $k=1, \dotsc,m$.
  By density it
  suffices to show that there is $\eps >0$ such that $f \not \in
  \clco(\Delta_\eps(f) \cap P)$ where $P =
  \mbox{span}(t^{\lambda_n})_{n=1}^\infty \subset X$.
  To this end, let $s$ be a point
  satisfying $(1 + y_m)/2 < s < 1$.
  By the Bernstein inequality
  \cite[Theorem~3.2]{MR1415318}, there exists a constant $c =
  c(\Lambda, s)$ such that for any $p \in P$
  \begin{align*}
    \|p'\|_{[0,s]} \le c \|p\|_{[0,1]}.
  \end{align*}

  Since $f \in C[0,1]$ there exists $\delta > 0$
  such that for all $x,y \in [0,1]$
  \begin{equation*}
    |x - y| < \delta \implies
    |f(x) - f(y)| < 1.
  \end{equation*}
  By choosing $\delta$ smaller if necessary
  we may assume that $c \delta < 1/2$
  and that $y_m + \delta/2 < s$.
  Let $I_{k,\delta} := (y_k - \delta/2,y_k + \delta/2)$.
  Note that $f$ does not change sign on any $I_{k,\delta}$.

  Put $I_\delta := \cup_{k=1}^m I_{k,\delta}$, and
  $M := \sup \{|f(y)|: y \in [0,1] \setminus I_\delta\}$.
  Since $[0,1] \setminus I_\delta$ is compact and
  since $f$ is continuous, the value $M$ is attained and thus $M < 1$.
  Let $0 < \eps < \min\{1/(2m), 1 - M, 1/4\}$. Then
  \begin{align*}
    |f(x)| \ge 1 - \eps \implies x
    \in I_\delta.
  \end{align*}
  Assume that $p \in \Delta_{\eps}(f) \cap P$.
  Since $\|f - p\| \ge 2 - \varepsilon$ the norm
  is attained on $I_\delta$.
  Therefore there exist $k$ and $x \in I_{k,\delta}$ such that
  \begin{equation*}
    |f(x) - p(x)| \ge 2 - \varepsilon.
  \end{equation*}
  Since $|f(x)| \ge 1 - \varepsilon$ and
  $f$ does not change sign on $I_{k,\delta}$
  we must have $|f(x) - f(y_k)| \le \varepsilon$
  hence
  \begin{align*}
    |f(y_k) - p(y_k)|
    &\ge
      |f(x) - p(x)| - |f(y_k) - f(x)| - |p(x) - p(y_k)| \\
      &\ge
      2 - 2\varepsilon - \|p'_i\|_{[0,s]}|x - y_k|
      > 3/2 - c\delta > 1.
  \end{align*}
  Now, let $n \in \mathbb{N}$ and
  $p_1, \dotsc, p_n \in \Delta_{\eps}(f) \cap P$.
  Find $r \in \mathbb{N}$ such that
  $(r-1)m < n \le r m$.
  By the pigeonhole principle,
  there is an interval $I_{j,\delta}$ where at least $r$
  of the polynomials $(p_i)_{i=1}^n$ satisfy
  $|f(y_j) - p_i(y_j)| > 1$.
  Put
  \begin{equation*}
    L := \{i \in \{1, \dotsc, n\}:
    |f(y_j) - p_i(x)| > 2 - 2\eps, x \in I_{j,\delta}\}.
  \end{equation*}
  We get that
  \begin{align*}
    |f(y_j) - \frac{1}{n}\sum_{i=1}^n p_i(y_j)|
    &\ge |f(y_j) - \frac{1}{n} \sum_{i \in L} p_i(y_j)|
      - \frac{1}{n}\sum_{i \notin L} |p_i(y_j)|\\
    &>
      1
      -
      \frac{1}{n}\sum_{i \notin L} 1
      \ge
      \frac{r}{n}
      \ge \frac{1}{m}
      > \eps.
  \end{align*}
  Hence $f \notin \clco(\Delta_{\eps}(f) \cap P)$.
\end{proof}
\begin{thm}\label{thm:muntz-DaugDeltaChar}
  Let $X$ be a M{\"u}ntz space $M_0(\Lambda)$ with $\lambda_1 \ge 1.$ The
  following assertions for $f\in S_X$ are equivalent:
  \begin{enumerate}
  \item\label{item:muntz-Daug-1}
    $f$ is a Daugavet point;
  \item\label{item:muntz-Daug-2}
    $f$ is a $\Delta$-point;
  \item\label{item:muntz-Daug-3}
    $\|f\| = |f(1)|.$
  \end{enumerate}
\end{thm}

\begin{proof}
  \ref{item:muntz-Daug-1} $\Rightarrow$ \ref{item:muntz-Daug-2} is trivial,
  \ref{item:muntz-Daug-2} $\Rightarrow$ \ref{item:muntz-Daug-3}
  follows from Proposition~\ref{prop:muntz-delta},
  and \ref{item:muntz-Daug-3} $\Rightarrow$ \ref{item:muntz-Daug-1}
  is Proposition~\ref{f at the end 1 is Daugavet point}.
\end{proof}

\section{Stability results}
\label{sec:stability-results}
Let us recall that a norm $N$ on $\mathbb{R}^2$
is \emph{absolute} if
\begin{equation*}
  N(a,b) = N(|a|,|b|) \quad \text{for all }
  (a,b) \in \mathbb{R}^2
\end{equation*}
and \emph{normalized} if
\begin{equation*}
  N(1,0) = N(0,1) = 1.
\end{equation*}
If $X$ and $Y$ are Banach spaces and $N$ is an absolute
normalized norm on $\mathbb{R}^2$, then we denote by
$X \oplus_N Y$ the product space $X \times Y$ with
the norm $\|(x, y)\|_N = N(\|x\|,\|y\|)$.

In this section we analyze how
$\Delta$- and Daugavet-points behave while taking direct sums
with absolute normalized norm $N$. First note a useful result that
simplifies the proofs.
\begin{lem}\label{lem:conv-sum-average}
  Let $m \in \mathbb{N}$. Then for all $\varepsilon > 0$,
  all $\lambda_i > 0$ with $\sum_{i=1}^m \lambda_i = 1$,
  there exists $n \in \mathbb{N}$, $k_1,\ldots,k_m \in \mathbb{N}$
  such that
  \begin{equation*}
    \sum_{i=1}^m \left| \lambda_i - \frac{k_i}{n}\right| < \varepsilon
    \qquad \mbox{and} \qquad
    \sum_{i=1}^m k_i = n.
  \end{equation*}
  In particular, every convex combination
  of elements in a normed vector space
  can be approximated arbitrarily
  well with an average of the same elements
  (each repeated $k_i$ times).
  Furthermore, given two such convex combinations,
  we can express them both as an average of
  the same number of elements.
\end{lem}

\begin{proof}
  By Dirichlet's approximation theorem given $N \in \mathbb{N}$
  there exist integers $k_1, \ldots, k_m$
  and $1 \le n \le N$ such that
  \begin{equation*}
    \left| \lambda_i - \frac{k_i}{n} \right|
    \le \frac{1}{n N^{1/m}}.
  \end{equation*}
  Then
  \begin{align*}
    \left| n - \sum_{i=1}^m k_i \right|
    &=
      n
      \left|
      \sum_{i=1}^m \lambda_i - \sum_{i=1}^m \frac{k_i}{n}
      \right|
      \le
      n
      \sum_{i=1}^m \frac{1}{n N^{1/m}}
      = \frac{m}{N^{1/m}}.
  \end{align*}
  By just choosing $N$ so large that $N^{-1/m} < \varepsilon$
  and $m N^{-1/m} < 1$ we get the desired conclusion.
  By choosing $\varepsilon > 0$ smaller if necessary
  we can make sure that $k_i \ge 0$ for $i = 1,\dotsc,m$.
\end{proof}

It is not hard to see that if a Banach space $X$ has a $\Delta$-point,
then $X\oplus_N Y$ has a $\Delta$-point too for any Banach space
$Y$. Moreover, if $x\in \Delta_X$ and $y\in \Delta_Y$, then
for all $a,b\geq 0$ with $N(a,b) = 1$ we have $(ax,by)\in
\Delta_Z$ (see the proof of Theorem~\ref{sum has convDLD2P}).
This implies that if $X$ and $Y$ both have the DLD2P
then $X \oplus_N Y$ has the DLD2P for any absolute normalized
norm $N$ on $\mathbb{R}^2$ (this was shown in
\cite{Zbl 1071.46015} using slices). In contrast,
there are absolute normalized norms $N$ for which the space $X\oplus_N
Y$ has no Daugavet-points. Therefore there even exists a space where
every unit sphere point is a $\Delta$-point, but none of them are
Daugavet-points. However, the matter of the existence of Daugavet-points
in direct sums is more complex as can be seen from the
following propositions.

\begin{defn}
  An absolute normalized norm $N$ on $\mathbb{R}^2$
  is \emph{positively octahedral} \cite{HLN} if
  there exist
  $a, b \ge 0$ such that $N(a,b) = 1$, and
  \begin{equation*}
    N\left( (0,1) + (a,b) \right) = 2
    \quad \mbox{and} \quad
    N\left( (1,0) + (a,b) \right) = 2.
  \end{equation*}
\end{defn}

\begin{prop}\label{sum has Daugavet point}
  Let $N$ be a positively octahedral norm on $\mathbb{R}^2$.
  If $X$ and $Y$ are two Banach spaces that both
  have Daugavet-points, then
  $X \oplus_N Y$ also has a Daugavet-point.
\end{prop}

\begin{proof}
  Let $X$ and $Y$ be Banach spaces and
  $N$ a positively octahedral absolute normalized norm.
  Denote $Z=X\oplus_N Y$. Let
  $x\in S_X$ and $y\in S_Y$ be
  Daugavet points.
  Since $N$ is positively octahedral,
  there exist $a, b \ge 0$ such that
  $N(a,b) = 1$ and $N((a,b) + (c,d)) = 2$
  for every $c, d \ge 0$ with $N(c,d) = 1$.
  We will show that $(ax, by)$ is a Daugavet point.

  Let $\nu := N(1,1)$.
  Fix $\varepsilon > 0$, $(u,v) \in S_{Z}$, and
  $\delta > 0$. First consider the case $u\neq 0$ and $v\neq 0$.
  Since
  $u/\|u\| \in \clco \Delta_{\varepsilon/\nu}^X (x)$ and
  $v/\|v\| \in \clco \Delta_{\varepsilon/\nu}^Y (y)$,
  we have $x_1, \dots, x_m \in \Delta_{\varepsilon/\nu}^X(x)$
  and $y_1, \dots, y_m \in \Delta_{\varepsilon/\nu}^Y(y)$
  such that (here we use Lemma~\ref{lem:conv-sum-average}
  to get the same number of vectors in $X$ and $Y$)
  \begin{equation*}
    \Big\Vert
    \frac{u}{\|u\|} - \frac{1}{m} \sum_{i=1}^{m} x_i
    \Big\Vert
    < \delta
    \quad
    \text{and}
    \quad
    \Big\Vert
    \frac{v}{\|v\|} - \frac{1}{m} \sum_{i=1}^{m} y_i
    \Big\Vert
    < \delta.
  \end{equation*}
  Therefore
  \begin{align*}
    &\Big\Vert
      (u,v) - \frac{1}{m} \sum_{i=1}^{m}
      \big( \|u\| x_i, \|v\| y_i \big)
      \Big\Vert_N \\
    &=
      N\Big(
      \|u\|\Big\Vert
      \frac{u}{\|u\|} - \frac{1}{m} \sum_{i=1}^{m} x_i
      \Big\Vert,
      \|v\|\Big\Vert
      \frac{v}{\|v\|} - \frac{1}{m} \sum_{i=1}^{m} y_i
      \Big\Vert
      \Big) \\
    &\le
      \delta N\big(\|u\|, \|v\|\big) = \delta.
  \end{align*}
  Note that
  \begin{equation*}
    \| ax - \|u\| x_i\| \ge a + \|u\| - \varepsilon/\nu
  \end{equation*}
  and
  \begin{equation*}
    \| by - \|v\| y_i\| \ge b + \|v\| - \varepsilon/\nu
  \end{equation*}
  by the reverse triangle inequality.
  This implies that
  $\big( \|u\|x_i, \|v\|y_i\big)\in \Delta_\varepsilon^Z (ax, by)$
  since
  \begin{align*}
    &N\big(
      \left\| ax - \|u\|x_i \right\|,
      \left\| by - \|v\|y_i \right\|
      \big) \\
    &\geq
      N\big(
      a + \|u\| - \varepsilon/\nu,
      b + \|v\| - \varepsilon/\nu
      \big) \\
    &\geq
      N\big( a + \|u\|, b + \|v\| \big)
      - N \big( \varepsilon/\nu, \varepsilon/\nu \big)
      = 2 - \varepsilon.
  \end{align*}
  If $u = 0$ or $v = 0$, the proof is simpler.
\end{proof}

\begin{defn}
  We will say that an
  absolute normalized norm $N$ on $\mathbb{R}^2$
  has property $(\alpha)$ if for every
  $c,d\geq 0$ with $N(c,d) = 1$,
  there exist $\varepsilon>0$ and neighbourhood
  $W$ of $(c,d)$ in $\mathbb{R}^2$ such that:
  \begin{itemize}
  \item
    if $a,b\geq 0$ satisfies $N(a,b)=1$ and
    \begin{equation*}
      N( (a,b) + (c,d) ) \ge 2 - \varepsilon,
    \end{equation*}
    then $(a,b) \in W$;
  \item
    either
    $\sup_{(a,b) \in W} a < 1$ or $\sup_{(a,b) \in W} b < 1$.
  \end{itemize}
\end{defn}

\begin{rem}
  The $\ell_p$-norm, $1 < p < \infty$, on $\mathbb{R}^2$
  has property ($\alpha$).

  Given $c, d \ge 0$ with $\|(c,d)\|_p = 1$
  for all $\delta > 0$ there exists $\varepsilon > 0$
  such that for all $(a,b)$ with
  $\|(a,b)\|_p \le 1$ and
  $\|(a,b) + (c,d)\|_p \ge 2 - \varepsilon$
  we have $(a,b) \in B((c,d),\delta) =: W$.
  Choosing $\delta$ small enough we have
  either $\sup_{(a,b) \in W} a < 1$ or
  $\sup_{(a,b) \in W} b < 1$.

  Similarly, any strictly convex absolute normalized norm $N$
  on $\mathbb{R}^2$ has property ($\alpha$).
\end{rem}

\begin{prop}\label{sum no Daugavet points}
  Let $X$ and $Y$ be Banach spaces and $N$
  an absolute normalized norm on $\mathbb{R}^2$
  with property ($\alpha$).
  Then $X \oplus_N Y$ has no Daugavet points.
\end{prop}

\begin{proof}
  Let $X$ and $Y$ be Banach spaces and $N$
  an absolute normalized norm on $\mathbb{R}^2$
  with property ($\alpha$).
  Denote $Z=X\oplus _N Y$
  and let $z = (x,y) \in S_Z$.

  Let $(c,d) = (\|x\|,\|y\|)$.
  From the definition of property ($\alpha$)
  there exists $\varepsilon > 0$ and
  a neighbourhood $W$ of $(c,d)$.
  Without loss of generality we may assume that
  $\sup_{(a,b) \in W} a < 1$ since the case
  $\sup_{(a,b) \in W} b < 1$ is similar.
  Choose $\delta > 0$ such that
  $\sup_{(a,b) \in W} \le 1 - \delta$.

  Assume that $(u,v) \in \Delta_\varepsilon(z)$.
  Then
  \begin{align*}
    2 - \varepsilon \le
    N(\|u - x\|, \|v - y\|)
    \le
    N(\|u\| + \|x\|, \|v\| + \|y\|),
  \end{align*}
  hence $(\|u\|, \|v\|) \in W$
  from property ($\alpha$).
  In particular, $\|u\| \le 1 - \delta$.

  Let $w \in S_X$ and consider $(w,0) \in S_Z$.
  Given
  $(x_1, y_1), \dots, (x_n, y_n) \in \Delta_\varepsilon(z)$
  we have
  $\|x_i\| \le 1 - \delta$ for each $i = 1,\dotsc,n$
  and
  \begin{align*}
    \Big\Vert (w,0) - \frac{1}{n}\sum_{i=1}^{n} (x_i,y_i) \Big\Vert_N
    &\ge \Vert w - \frac{1}{n}\sum_{i=1}^{n} x_i\Vert
      \ge \|w\| - \frac{1}{n}\sum_{i=1}^{n}\|x_i\|
    \\
    &\ge 1 - \frac{1}{n}\sum_{i=1}^{n} (1 - \delta)
      = \delta.
  \end{align*}
  Using Lemma~\ref{lem:conv-sum-average} we see that
  this means that $(w,0) \notin \clco\Delta_{\varepsilon} (z)$,
  and we conclude that $z$ is not a Daugavet-point.
\end{proof}

\begin{eks}\label{exmp:Daug-vs-Delta}
  Consider the space $X = C[0,1] \oplus_2 C[0,1]$.

  $C[0,1]$ has the Daugavet property and in particular
  the DLD2P, hence $X$ has the DLD2P \cite[Theorem~3.2]{Zbl 1071.46015}.
  But, by Proposition~\ref{sum no Daugavet points},
  $X$ has no Daugavet-points even though
  every $x \in S_X$ is a $\Delta$-point.
\end{eks}

\section{The convex DLD2P}
\label{sec:convex-dld2p}
In this last section we consider Banach spaces $X$
with the property that $B_X = \clco(\Delta)$.
We show that this property is a diameter two property
that differs from the already known diameter two properties.
We also give examples of spaces with this new property.
\begin{defn}
  Let $X$ be a Banach space. If $B_X = \clco(\Delta)$,
  then we say that $X$ has the
  \emph{convex diametral local diameter two property (convex DLD2P)}.
\end{defn}

\begin{prop}
  Let $X$ be a Banach space. If $X$ has the convex DLD2P, then $X$ has
  the LD2P.
\end{prop}

\begin{proof}
  Let $x^* \in S_{X^*}$, $\eps > 0$, and consider the slice
  \begin{equation*}
    S(x^*,\eps) = \{x \in B_X: x^*(x) > 1- \eps\}.
  \end{equation*}
  Pick some $\hat{x} \in S(x^*,\eps/4)$. Choose
  $(x_i)_{i=1}^n \subset \Delta$ and a convex combination
  $x:= \sum_{i=1}^n \lambda_i x_i $ with
  $\|x - \hat{x}\| < \eps/4$. Now at least one of the $x_i$'s must be
  in $S(x^*,\eps/2)$ otherwise
  \begin{equation*}
    x^*(x) = \sum_{i=1}^n \lambda_i x^*(x_i)
    < \sum_{i=1}^n \lambda_i (1 - \eps/2)
    < 1 - \eps/2
  \end{equation*}
  which contradicts the fact that $\hat{x} \in S(x^*,\eps/4)$ and
  $\|\hat{x} - x\| < \eps/4$.
  Now let $x_k$ be one of the $x_i$'s which are in $S(x^*,\eps/2)$
  and use the same idea as above to produce some
  $y \in \Delta_\eps(x_k)$ such that $y \in S(x^*,\eps)$.
  Since $x_k \in S(x^*,\eps/2) \subset S(x^*,\eps)$ and
  $\|x_k - y\| > 2 - \eps$ we are done.
\end{proof}

\begin{prop}\label{prop:CK-convDLD2P}
  If $K$ is an infinite compact Hausdorff space,
  then $C(K)$ has the convex DLD2P.
\end{prop}

\begin{proof}
  We only need to show that
  $S_{C(K)} \subset \clco \Delta$.
  Let $f \in C(K)$ with $\|f\| = 1$.
  If $|f(x)| = 1$ for some limit point
  of $K$, then $f \in \Delta$ by
  Theorem~\ref{thm:CK-Daug-Delta-limit}.
  Assume that $|f(x)| < 1$ for every
  limit point of $K$ and
  let $x_0$ be a limit point of $K$.

  Let $\varepsilon > 0$ and choose a neighbourhood
  $U$ of $x_0$ such that $|f(x) - f(x_0)| < \varepsilon$
  for every $x \in U$.
  We use Urysohn's lemma to find a function
  $\eta : K \to [0,1]$ such that
  $\eta(x_0) = 1$ and $\eta = 0$ on $K \setminus U$.
  Define
  \begin{align*}
    f^+(x) &:= (1-\eta(x)) f(x) + \eta(x)(1), \\
    f^-(x) &:= (1-\eta(x)) f(x) + \eta(x)(-1).
  \end{align*}
  Then $f^{\pm} \in B_{C(K)}$ and both are
  in $\Delta$ by Theorem~\ref{thm:CK-Daug-Delta-limit}.
  Let $\lambda := \frac{1+f(x_0)}{2}$ and consider
  \begin{equation*}
    g(x) := \lambda f^+(x) + (1-\lambda)f^{-}(x).
  \end{equation*}
  Then
  \begin{equation*}
    g(x) =
    \begin{cases}
      f(x) & x \in K \setminus U, \\
      (1-\eta(x))f(x) + \eta(x)f(x_0) & x \in U.
    \end{cases}
  \end{equation*}
  We get
  \begin{equation*}
    \|g - f\|
    \le \max_{x \in U} |\eta(x)(f(x) - f(x_0))|
    < \varepsilon.
  \end{equation*}
  Since $\varepsilon > 0$ was arbitrary we
  get that $f \in \clco \Delta$.
\end{proof}

\begin{cor}\label{ell_infty has convex DLD2P}
  Both $c = C([0,\omega])$ and $\ell_\infty = C(\beta \mathbb{N})$
  have the convex DLD2P.
\end{cor}

\begin{rem}\label{c_0 Delta empty}
  In $c$ the points in $\Delta$ are exactly
  the sequences with limit $1$ or $-1$.
  For $\ell_\infty$ we have that
  $\Delta$ consists of all sequences
  $(x_n) \in \ell_{\infty}$ such that
  $|\lim_{\mathcal{U}} x_n| = 1$, where
  $\mathcal{U}$ is a non-principal ultrafilter
  on $\mathbb{N}$.
  In particular, none of these spaces
  have the DLD2P.

  For $c_0$ we have $\Delta = \emptyset$
  since $\Delta$-points in $c_0$ have to
  be $\Delta$-points in $\ell_\infty$
  by Lemma~\ref{lem:Daug-ai-ideal}.
  Hence the convex DLD2P is not inherited from the bidual
  unlike the LD2P.
  The convex DLD2P is also not inherited by
  subspaces of codimension 1, since
  $c_0$ is of codimension 1 in $c$.
\end{rem}

Considering the facts that $\ell_\infty$ does not have the DLD2P and
$c_0$ has the LD2P, Remark~\ref{c_0 Delta empty}, and
Corollary~\ref{ell_infty has convex DLD2P},
we can conclude that the convex DLD2P
is a new diameter-2 property, different from the ones observed so far.

\begin{cor}\label{cor:convDLD2Pisnew}
  Let $X$ be a Banach space. Then
  \begin{align*}
    DLD2P \implies \text{convex } DLD2P \implies LD2P,
  \end{align*}
  where the implications cannot be reversed.
\end{cor}

Our next aim is to show that M\"{u}ntz spaces also have the convex
DLD2P.

\begin{thm}\label{Mutnz space has cDLD2P}
  Let $X = M(\Lambda)$ or $X = M_0(\Lambda)$ be a M\"{u}ntz space.
  Then $X$ has the convex DLD2P.
\end{thm}

\begin{proof}
  It is enough to show that $S_X \subset \clco{\Delta}$.
  Since $P := \linspan\{t^{\lambda_n}\}$ is dense in $X$,
  it is enough to show that if
  $f \in B_P$ with $\|f\| = 1 - s$ for some $0 < s < 1$,
  then $f \in \conv \Delta$.
  To this end, given $n \in \mathbb{N}$ we define
  \begin{equation*}
    f^{+}_n(x) = f(x) + (1 - f(1)) x^{\lambda_n}
  \end{equation*}
  and
  \begin{equation*}
    f^{-}_n(x) = f(x) - (1 + f(1)) x^{\lambda_n}.
  \end{equation*}
  From Proposition~\ref{f at the end 1 is Daugavet point}
  we see that $f^{\pm}_n$ are candidates for being
  $\Delta$-points since
  \begin{equation*}
    f^{\pm}_n(1) = f(1) \pm (1 \mp f(1)) = \pm 1.
  \end{equation*}
  If we define $\mu = \frac{f(1)+1}{2}$, that is,
  $2\mu - 1 = f(1)$, we have a convex combination
  \begin{equation*}
    \mu f^+_n(x) + (1 - \mu)f_n^{-}(x)
    = f(x) + \bigl( 2\mu - 1 - f(1) \bigr) x^{\lambda_n}
    = f(x).
  \end{equation*}
  We need to show that when $n$ is large enough we have
  $f^{\pm}_n \in S_P$.

  Since $f \in P$ we can write
  \begin{equation*}
    f(x) = \sum_{k=0}^m a_k x^{\lambda_k}.
  \end{equation*}
  Now, $f$, $f'$, and $f''$ are all generalized polynomials
  so by Descartes' rule of signs,
  see e.g. \cite[Theorem~3.1]{Jameson},
  they only have a finite number of zeros on $(0,1]$.
  Hence there exists $t_0 \in (0,1)$
  such that neither $f'$ nor $f''$ changes sign
  on $(t_0,1)$.
  Without loss of generality we may assume that
  $f' < 0$ on $(t_0,1)$.
  (If $f' > 0$ on $(t_0,1)$ we consider $-f$.)

  There exists $N$ such that
  \begin{equation}\label{eq:1}
    t_0^{\lambda_n} < s/2 \quad
    \text{for}\ n > N.
  \end{equation}
  For $n > N$ we get
  \begin{equation*}
    |f^-_n(x)| \le 1 - s + (1+f(1))s/2 \le 1
  \end{equation*}
  on $[0,t_0]$ and on $[t_0,1]$ we have
  \begin{equation*}
    \frac{d}{dx}(f^{-}_n(x))
    = f'(x) - \lambda_n (1+f(1)) x^{\lambda_n - 1}
    < 0.
  \end{equation*}
  We have $|f^{-}_n(x)| \le 1$
  in both endpoints of $[t_0,1]$.
  Hence $\|f^-_n\| \le 1$.

  It remains to find $n > N$ such that also
  $f^+_n \in S_P$. We consider two cases.

  \textbf{Case I:}
  Assume there exists $0 < t_0 < 1$ such that
  $f' < 0$ and $f'' > 0$ on $(t_0,1)$.
  For $n > N$ we have $d^2/dx^2(f^+_n) > 0$ on $(t_0,1)$,
  hence $f^+_n$ is convex on $[t_0,1]$
  and (by using \eqref{eq:1})
  \begin{equation*}
    \|f^+_n\| \le \max(f^+_n(t_0),f^+_n(1))
    \le \max(1 - s + (1-f(1))t_n^{\lambda_n},1) \le 1
  \end{equation*}
  since also $f^+_n(x) > f(x) \ge -1$ for all $x \in [0,1]$.

  \textbf{Case II:}
  Assume there exists $0 < t_0 < 1$ such that
  $f' < 0$ and $f'' < 0$ on $[t_0,1]$.

  Let $\delta := f(t_0) - f(1) > 0$.
  Define
  \begin{equation*}
    t_n := \sqrt[\lambda_n]{1-\frac{\delta}{1-f(1)}},
  \end{equation*}
  that is
  \begin{equation*}
    t_n^{\lambda_n} = \frac{1-f(1) - \delta}{1-f(1)}
  \end{equation*}
  Note that $t_n \to 1$.

  Write $g_n(x) = (1-f(1))x^{\lambda_n}$.
  Then $g'_n(x) = (1-f(1))\lambda_n x^{\lambda_n - 1}$
  and
  \begin{align*}
    g'_n(t_n)
    &= (1-f(1)) \lambda_n
      \frac{1-f(1)-\delta}{1-f(1)}
      \left(\frac{1-f(1)-\delta}{1-f(1)}\right)^{-1/\lambda_n}
    \\
    &= \lambda_n (1-f(1)-\delta)
      \left(\frac{1-f(1)-\delta}{1-f(1)}\right)^{-1/\lambda_n}.
  \end{align*}
  Note that $g'_n(t_n) \to \infty$
  (since we assume that $\sum_{n=1}^\infty \lambda_n^{-1} < \infty$).
  Let $M := \max_{x \in [t_0,1]} |f'(x)|$.
  Choose $n > N$ such that $t_0 < t_n < 1$
  and
  \begin{equation*}
    g'_n(t_n) > M.
  \end{equation*}
  Then for $x \in [t_n,1]$ we have
  \begin{equation*}
    \frac{d}{dx}(f^{+}_n(x))
    = f'(x) + \lambda_n (1-f(1)) x^{\lambda_n - 1}
    > -M + g'_n(t_n) > 0
  \end{equation*}
  hence $f^{+}_n(x) \le f^{+}_n(1)$ on $[t_n,1]$.

  For $x \in [t_0,t_n]$ we get
  \begin{align*}
    f^+_n(x)
    &= f(x) + g_n(x)
      \le f(1) + \delta + (1-f(1))t_n^{\lambda_n} \\
    &= f(1) + \delta + (1-f(1)-\delta) \le 1.
  \end{align*}
  While on $[0,t_0]$ we have, by using \eqref{eq:1},
  \begin{equation*}
    |f^{+}_n(x)| \le \|f\| + 2 \cdot s/2 \le 1.
  \end{equation*}
  Hence $\|f^+_n\| \le 1$.
\end{proof}

It is known that given Banach spaces $X$ and $Y$, they have the
Daugavet property if and only if $X\oplus_1 Y$ or $X\oplus_\infty Y$
has Daugavet property
(see \cite[Lemma~2.15]{MR1621757} and \cite[Corollary~5.4]{BKSW}).
For the DLD2P we have that for any absolute normalized norm
on $\mathbb{R}^2$, both $X$ and $Y$ have the DLD2P if and only
if $X \oplus_N Y$ has the DLD2P \cite[Theorem~3.2]{Zbl 1071.46015}.
The following theorem shows that the convex DLD2P also behaves well
under direct sums.
\begin{thm}\label{sum has convDLD2P}
  Let $N$ be an absolute normalized norm on $\mathbb{R}^2$.
  If $X$ and $Y$ have the convex DLD2P,
  then $X\oplus_N Y$ has the convex DLD2P.
\end{thm}
\begin{proof}
  Assume that $X$ and $Y$ are Banach spaces with the convex DLD2P.
  Denote $Z = X \oplus_N Y$.

  \textbf{Claim:} If $a, b \ge 0$ with $N(a,b) = 1$,
  $x \in \Delta_X$, and $y \in \Delta_Y$,
  then $(ax, by) \in \Delta_Z$.

  \textbf{Proof of claim.}
  Let $\varepsilon > 0$ and $0 < \gamma < \varepsilon$.
  Since $x \in \Delta_X$ and $y \in \Delta_Y$,
  we have $x_1, \dotsc, x_m \in \Delta^X_\varepsilon(x)$ and
  $y_1, \dotsc, y_m \in \Delta^Y_\varepsilon(y)$
  such that
  (using Lemma~\ref{lem:conv-sum-average})
  \begin{equation*}
    \Big\| x - \frac{1}{m}\sum_{i=1}^{m} x_i \Big\|
    < \gamma
    \quad \text{and} \quad
    \Big\| y - \frac{1}{m}\sum_{i=1}^{m} y_i\Big\|
    < \gamma.
  \end{equation*}
  Note that
  \begin{align*}
    \Big\|
    ( ax, by )
    - \frac{1}{m}\sum_{i=1}^{m} \big(a x_i, b y_i \big)
    \Big\|_N
    &=
      N\Big(
      a \Big\| x - \frac{1}{m}\sum_{i=1}^{m} x_i \Big\|,
      b \Big\| y - \frac{1}{m}\sum_{i=1}^{m} y_i \Big\|
      \Big) \\
    &\leq
      N(\gamma a, \gamma b)
      = \gamma N(a, b)
      = \gamma,
  \end{align*}
  and
  \begin{align*}
    \big\|
    ( ax, by ) - ( a x_i, b y_i )
    \big\|_N
    &=
      N(a \|x - x_i\|, b \|y - y_i\|) \\
    &\geq
      N(a (2-\varepsilon), b (2-\varepsilon)) \\
    &=
      (2-\varepsilon) N(a, b)
      = 2-\varepsilon.
  \end{align*}
  This concludes the proof of the claim.

  Now let $(x,y) \in S_Z$.
  We will show that $(x,y) \in \clco\Delta_Z$.

  Let $\delta >0$.
  First consider the case $x\neq 0$ and $y\neq 0$. Then
  $\frac{x}{\|x\|} \in \clco \Delta_X$ and
  $\frac{y}{\|y\|} \in \clco\Delta_Y$ by the assumption;
  hence there are $x_1,\dots,x_n \in \Delta_X$
  and $y_1,\dots,y_n \in \Delta_Y$ such that
  (here we use Lemma~\ref{lem:conv-sum-average} again)
  \begin{equation*}
    \Big\Vert \frac{x}{\|x\|} - \frac{1}{n}\sum_{i=1}^{n} x_i\Big\Vert < \delta
    \quad \text{and} \quad
    \Big\Vert \frac{y}{\|y\|} - \frac{1}{n}\sum_{i=1}^{n} y_i\Big\Vert < \delta.
  \end{equation*}
  By the claim above we have $(\|x\|x_i, \|y\|y_i) \in \Delta_Z$.
  All that remains is to note that
  \begin{align*}
    &\Big\Vert
      (x,y)
      -
      \frac{1}{n}\sum_{i=1}^{n} \big( \|x\|x_i, \|y\|y_i \big)
      \Big\Vert_N \\
    &=
      N\Big(
      \|x\|\Big\Vert \frac{x}{\|x\|} - \frac{1}{n}\sum_{i = 1}^{n} x_i\Big\Vert,
      \|y\|\Big\Vert \frac{y}{\|y\|} - \frac{1}{n}\sum_{i = 1}^{n} y_i\Big\Vert
      \Big) \\
    &\leq
      N\big( \delta \|x\|, \delta \|y\| \big)
    = \delta N\big( \|x\|, \|y\| \big) = \delta.
  \end{align*}
  Now consider the case where $y=0$
  (a similar argument holds for the case $x=0$).
  We have
  \begin{equation*}
    \|(x,0)\|_N = N(\|x\|,0) = \|x\|,
  \end{equation*}
  so that $(x,0) \in \clco\Delta_Z$
  follows from $x\in \clco\Delta_X$
  since the claim above shows that
  $(x_i,0) \in \Delta_Z$ when $x_i \in \Delta_X$.
\end{proof}

\begin{rem}
  Let $X$ and $Y$ be Banach spaces.
  If $X$ has the convex DLD2P and $N$ is the $\ell_\infty$-norm,
  then $X\oplus_N Y$ has the convex DLD2P.
\end{rem}
Although we have mostly settled the results about the question
whether the direct sum with absolute normalized norm has
a $\Delta$-point/a Daugavet-point/the convex DLD2P
(there are some norms left to look at in the Daugavet-point case),
the results about the components of a direct sum with a given
property having the same property, are all still unknown.
\begin{prob}
  Given $X\oplus_N Y$ with a
  $\Delta$-point/a Daugavet point/the convex DLD2P,
  does $X$ have a $\Delta$-point/a Daugavet point/the convex DLD2P?
\end{prob}
%
%

\bibliographystyle{amsplain}

\end{document}